\documentclass{amsart}
\usepackage{amsmath}
\usepackage{amssymb}
\usepackage{graphicx}
\usepackage{amscd}

\newtheorem{theorem}{Theorem}
\theoremstyle{plain}

\newtheorem{lemma}{Lemma}

\newtheorem{proposition}{Proposition}

\numberwithin{equation}{section}

\begin{document}
\title[SLLN for random fields]{On general strong laws of large numbers for
fields of random variables}
\author{Cheikhna Hamallah Ndiaye $^{*}$}
\email{chndiaye@ufrsat.org}
\address{LERSTAD, Universit\'{e} de Saint-Louis $^{*}$}
\address{LMA - Laboratoire de Math\'ematiques Appliqu\'ees \\
Universit\'e Cheikh Anta Diop BP 5005 Dakar-Fann S\'en\'egal}

\author{ Gane Samb LO $^{**}$}
\email{ganesamblo@ufrsat.org}
\address{$^{**}$ Laboratoire de Statistiques Th\'eoriques et Appliqu\'ees (LSTA) \\Universit\'e Pierre et Marie Curie (UPMC) France}
\address{LERSTAD, Universit\'{e} de Saint-Louis.}
\urladdr{www.lsta.upmc.fr/gslo}

\begin{abstract}
A general method to prove strong laws of large numbers for random fields is
given. It is based on the H\'{a}jek - R\'{e}nyi type method presented in Nosz%
\'{a}ly and T\'{o}m\'{a}cs \cite{noszaly} and in T\'{o}m\'{a}cs and L\'{i}bor 
\cite{thomas06}. Nosz\'{a}ly and T\'{o}m\'{a}cs \cite{noszaly} obtained a 
H\'{a}jek-R\'{e}nyi type maximal inequality for random fields using moments
inequalities. Recently, T\'{o}m\'{a}cs and L\'{i}bor \cite{thomas06} obtained a
H\'{a}jek-R\'{e}nyi type maximal inequality for random sequences based on
probabilities, but not for random fields. In this paper we present a 
H\'{a}jek-R\'{e}nyi type maximal inequality for random fields, using
probabilities, which is an extension of the main results of Nosz\'{a}ly and T%
\'{o}m\'{a}cs \cite{noszaly} by replacing moments by probabilities and a
generalization of the main results of T\'{o}m\'{a}cs and L\'{i}bor \cite%
{thomas06} for random sequences to random fields. We apply our results to
establishing a logarithmically weighted sums without moment assumptions and
under general dependence conditions for random fields.
\end{abstract}

\maketitle

\Large

\section{INTRODUCTION AND NOTATIONS}

\label{sec1}

We are concerned in this paper with strong laws of large numbers (SLLN) for
random fields. Although this type of problems is entirely settled for
sequence of independent real random variables (see for instance \cite{loeve}%
) and general strong laws of large numbers for dependent real random
variables based on H\'{a}jek-R\'{e}nyi types inequalities. But as for random
fields, they are still open. As a reminder, we recall that a family of
random elements $(X_{n})_{n\in T}$ is said to be a random field if the set
is endowed with a partial order ($\leq )$, not necessarily complete. For
example, and it is the case in this paper, $T$ may be $\mathbb{\mathbf{N}}%
^{d},$ where $d>1$ is an integer and $\mathbb{\mathbf{N}}$ is the set of
nonnegative integers. For such a real random field $(X_{n})_{n\in \textbf{N}^{d}},$ we intend to contribute to assessing the more general SLLN's, that
is finding general conditions under which there exists a real number $\mu $\ and
a family of normalizing positive numbers $(b_{n})_{n\in \textbf{N}^{d}},$ named here as a $d$-sequence, such that, for $S_{(0,...,0)}=0,$ and $%
S_{n}=\sum_{m\leq n}X_{m}$ for $n>\textbf{0},$ one has%
\begin{equation*}
S_{n}/b_{n}\rightarrow \mu, \\\ a.s.
\end{equation*}

\noindent In the case of random fields, the data may be heavily dependent
and then H\'{a}jek-R\'{e}nyi type maximal inequalities are needed to obtain
strong laws of large numbers, like in the real case. It seems that providing
such inequalities goes back to M\'{o}ricz \cite{moricz0} and Klesov \cite{klesov}%
. Based on such inequalities, many authors established strong laws of large
numbers such as Nguyen et al. \cite{nguyen}, T\'{o}m\'{a}cs \cite{thomas09}
, Lagodowski \cite{lagodowski}, Nosz\'{a}ly and T\'{o}m\'{a}cs \cite{noszaly}%
, M\'{o}ricz \cite{moricz}, Klesov \cite{klesov}, Fazekas et al. \cite%
{fazekas99}, Fazekas \cite{fazekas83}, \cite{fazekas98} and the literature
cited herein.

\bigskip

\noindent One of the motivations of finding general strong laws of large
numbers comes from that the finding, as proved by Cairoli \cite{caroli}, that
classical maximal probability inequalities for random sequences are not
valid in general for random fields. Besides, nonparametric estimation for
random fields or spatial processes was given increasing and simulated
attention over the last few years as a consequence of growing demands from
applied research areas (see for instance Guyon \cite{guyon}). This results
in the serious motivation to extend the H\'{a}jek-R\'{e}nyi type maximal
inequality for probabilities for random sequences, what the cited above
authors tackled.\newline

\noindent Our objective is to give a nontrivial generalization of some
fundamental results of these authors that will lead to positive answers to
classical and non solved SLLN's. Before a more precise formulation of our
problem, we need a few additional notation.\newline

\noindent From now on $d$ is a fixed positive integer. The elements of $%
\mathbb{\mathbf{N}}^{d}$ will be written in font bold like $\mathbf{n}$%
\textbf{\ }while their coordinate are written in the usual way like $\mathbf{%
n=}(n_{1},...,n_{d}).$ $\mathbb{\mathbf{N}}^{d}$ is endowed with the usual
partial ordering, that is $\mathbf{n}=(n_{1},...,n_{d})\leq \mathbf{m}%
=(m_{1},...,n_{d})$ if and only if or each $1\leq i\leq d,$ one has $%
n_{i}\leq m_{i}.$ Further $\textbf{m}<\textbf{n}$ \ means \ $\mathbf{m}\leq \mathbf{n}$ \ and
\ $\mathbf{n}\neq \mathbf{m}.$ We specially denote $(1,...,1)\equiv \mathbf{1}$ and $%
(0,...,0)\equiv \mathbf{0}$. All the limits, unless specification, are meant as 
\textbf{\ }$\mathbf{n}=(n_{1},...,n_{d})\rightarrow $ $\infty $, that is
equivalent to say that $n_{i}\rightarrow +\infty $ for each $1\leq i\leq d.$
To finish, any family of real numbers $(b_{\mathbf{n}})_{\mathbf{n}\in A}$
indexed by a subset $\mathbb{\mathbf{N}}^{d}$ is called a d-sequence. We
intensively use product type d-sequences. A d-sequence $(b_{\textbf{n}})_{\textbf{n}\in A}$ is
of product type if it may be written in the form 
\begin{equation*}
b_{\mathbf{n}}=\prod\limits_{1\leq i\leq d}b_{n_{i}}^{(i)}.
\end{equation*}%
This product type d-sequence is unbounded and nondecreasing if and only if
each sequence $b_{n_{i}}^{(i)}$ is unbounded and nondecreasing in $n_{i}.$
Now with these minimum notation, we are able to state the results of T\'{o}m%
\'{a}cs, L\'{i}bor and their co-authors.\newline

\noindent On one hand, it is known that the H\'{a}jek-R\'{e}nyi \ type
maximal inequality (see \cite{fazekas00}) is an important tool for proving
SLLN's for sequences. It is natural that Nosz\'{a}ly and T\'{o}m\'{a}cs \cite%
{noszaly} used a generalization of this result for random fields in order to
get SLLN's for such objects. They stated

\begin{proposition}
\label{prop1} Let \ $%
r$ \ be a positive real number, \ $a_{\mathbf{n}}$ \ be a nonnegative 
d-sequence. Suppose that \ $b_{\mathbf{n}}$ \ is a positive, nondecreasing 
d-sequence of product type. Then 
\begin{equation*}
\forall \mathbf{n}\in \mathbb{\mathbf{N}}^{d},\ \ \ \mathbb{E}({\max_{\mathbf{\ell}\leq 
\mathbf{n}}}|S_{\mathbf{\ell}}|^{r})\leq {\sum_{\mathbf{\ell}\leq \mathbf{n}}}a_{%
\mathbf{\ell}}\ \ 
\end{equation*}%
implies 
\begin{equation*}
\forall \mathbf{n}\in \mathbb{\mathbf{N}}^{d},\ \ \ \mathbb{E}\left( {\max_{\mathbf{\ell}%
\leq \mathbf{n}}}|S_{\mathbf{\ell}}|^{r}b_{\mathbf{\ell}}^{-r}\right) \leq 4^{d}{%
\sum_{\mathbf{\ell}\leq \mathbf{n}}}a_{\mathbf{\ell}}b_{\mathbf{\ell}}^{-r}.
\end{equation*}
\end{proposition}

\noindent From this, they were led to the following general SLLN for random
fields.

\begin{theorem}
\label{theo1} Let $a_{\mathbf{n}},b_{\mathbf{n}}$ be non-negative 
d-sequences and let $r>0$. Suppose that $b_{\mathbf{n}}$ is a positive,
nondecreasing, unbounded d-sequence of product type. Let us assume that 
\begin{equation*}
\sum_{\mathbf{n}}\frac{a_{\mathbf{n}}}{b_{\mathbf{n}}^{r}}<+\infty
\end{equation*}%
and

\begin{equation*}
\forall \mathbf{n}\in \mathbb{\mathbf{N}}^{d},\ \ \ \mathbb{E}\left( \max_{%
\mathbf{m}\leq \mathbf{n}}|S_{\mathbf{m}}|^{r}\right) \leq \sum_{\mathbf{m}%
\leq \mathbf{n}}a_{\mathbf{m}}.
\end{equation*}%
Then 
\begin{equation*}
\lim_{\mathbf{n}\rightarrow +\infty }\frac{S_{\mathbf{n}}}{b_{\mathbf{n}}}%
=0\ \ a.s.
\end{equation*}
\end{theorem}

\noindent On an other hand, T\'{o}m\'{a}cs and L\'{i}bor \cite{thomas06},
introduced a H\'{a}jek-R\'{e}nyi inequality for probabilities and,
subsequently, strong laws of large numbers for random sequences but not for
random fields. They obtained first :

\begin{theorem}
\label{theo2} Let $r$ be a positive real number, \ $a_{n}$ \ be a sequence
of nonnegative real numbers. Then the following two statements are
equivalent. \newline
(i) There exists \ $C>0$ \ such that for any \ ${n}\in \mathbb{\mathbf{N}}$
\ and any \ $\varepsilon >0$ 
\begin{equation*}
\mathbb{P}({\max_{{\ell}\leq {n}}}|S_{\ell}|\geq
\varepsilon )\leq C\varepsilon ^{-r}{\sum_{{\ell}\leq {n}}}a_{\ell}.
\end{equation*}%
(ii) There exists $C>0$ such that for any nondecreasing sequence $(b_{n})_{n \in \textbf{N}}$ of positive real numbers, for any ${n}\in 
\mathbb{\mathbf{N}}$ and any \ $\varepsilon >0$ 
\begin{equation*}
\mathbb{P}\left( {\max_{{\ell}\leq {n}}}%
|S_{\ell}|b_{\ell}^{-1}\geq \varepsilon \right) \leq C\varepsilon ^{-r}{\sum_{{\ell}%
\leq {n}}}a_{\ell}b_{\ell}^{-r}.
\end{equation*}
\end{theorem}

\noindent And next, they derived from it this SLLN.

\begin{theorem}
\label{theo3} Let $a_{n}\ and\ b_{n}$ are non-negative sequences of real
numbers and let $r>0$. Suppose that $b_{n}$ is a positive non-decreasing,
unbounded sequence of positive real numbers. Let us assume that 
\begin{equation*}
\sum_{n}\frac{a_{n}}{b_{n}^{r}}<+\infty
\end{equation*}%
and there exists $C>0$ such that for any ${n}\in \mathbb{\mathbf{N}}$ and
any $\varepsilon >0$ 
\begin{equation*}
\mathbb{P}\left( \max_{{m}\leq {n}}|S_{m}|\geq \varepsilon \right) \leq C\
\varepsilon ^{-r}\sum_{{m}\leq {n}}a_{m}.
\end{equation*}%
Then 
\begin{equation*}
\lim_{{n}\rightarrow +\infty }\frac{S_{n}}{b_{n}}=0\ \ a.s
\end{equation*}
\end{theorem}

\noindent As said previously, this paper aims at generalizing the previous
results in the following way. First, we give a random fields version for T\'{o}m%
\'{a}cs and L\'{i}bor \cite{thomas06} as a first generalization in Proposition %
\ref{prop2} . Next we show that our version of H\'{a}jek-R\'{e}nyi type maximal
inequality for probabilities for random fields is a generalization of that of
Nosz\'{a}ly and T\'{o}m\'{a}cs \cite{noszaly} and leads to a more general SLLN.%
\newline

\noindent We apply our method for logarithmically weighted sums without any
moment assumption and under general dependence conditions for random fields. This shows that
the generalization is not trivial.\newline

\noindent The paper is organized as follows. Section \ref{sec2} is devoted
to our main results, a H\'{a}jek-R\'{e}nyi type maximal inequality for
probabilities for random fields and automatically a strong law of large
numbers are given. Section \ref{sec3} includes their proofs. Section \ref%
{sec4} including applications and illustration of our results, concludes the
paper.

\section{RESULTS}

\label{sec2}

\noindent We first give a H\'{a}jek-R\'{e}nyi type maximal inequality for
probabilities for random fields, as an extension of Proposition 1 in 
Nosz\'{a}ly and T\'{o}m\'{a}cs \cite{noszaly} and of Theorem 2.1 in T\'{o}m\'{a}%
cs and L\'{i}bor \cite{thomas06}.

\begin{proposition}
\label{prop2} Let \ $r$ \ be a positive real number, \ $a_{\mathbf{n}}$ \ be
a nonnegative d-sequence. Suppose that \ $b_{\mathbf{n}}$ \ is a positive,
nondecreasing d-sequence of product type. Then the following two
statements are equivalent \newline
(i) There exists \ $C>0$ \ such that for any \ $\mathbf{n}\in \mathbb{%
\mathbf{N}}^{d}$ \ and any \ $\varepsilon >0$ 
\begin{equation*}
\mathbb{P}({\max_{\mathbf{\ell}\leq 
\mathbf{n}}}|S_{\mathbf{\ell}}|\geq \varepsilon )\leq C\varepsilon ^{-r}{\sum_{%
\mathbf{\ell}\leq \mathbf{n}}}a_{\mathbf{\ell}}\ \ 
\end{equation*}%
(ii) There exists \ $C>0$ \ such for any \ $\mathbf{n}\in \mathbb{\mathbf{N}}%
^{d}$ \ and any \ $\varepsilon >0$ 
\begin{equation*}
\mathbb{P}\left( {\max_{\mathbf{\ell}%
\leq \mathbf{n}}}|S_{\mathbf{\ell}}|b_{\mathbf{\ell}}^{-1}\geq \varepsilon \right)
\leq 4^{d} \ \ C\  \varepsilon ^{-r}\ {\sum_{{\mathbf \ell}\leq \mathbf{n}}}a_{\mathbf{%
\ell}}b_{{\bf \ell}}^{-r}\ \ 
\end{equation*}
\end{proposition}

\noindent We derive from this proposition  a general strong law of large
numbers for random fields which includes extensions of Theorem 3 in Nosz\'{a}%
ly and T\'{o}m\'{a}cs \cite{noszaly} and of Theorem 2.4 in T\'{o}m\'{a}cs
and L\'{i}bor \cite{thomas06}. But we need this lemma first.\newline

\begin{lemma}
\label{lemma1} (See Lemma 2 in Nosz\'{a}ly and T\'{o}m\'{a}cs \cite{noszaly}%
) Let \ $a_{\mathbf{n}}$ \ be a nonnegative d-sequence and let \ $b_{%
\mathbf{n}}$ \ be a positive, nondecreasing, unbounded d-sequence of
product type. Suppose that \ ${\sum_{\mathbf{n}}}\frac{a_{\mathbf{n}}}{b_{%
\mathbf{n}}^{r}}<+\infty $ \ with a fixed real \ $r>0.$ Then there exists a
positive, nondecreasing, unbounded d-sequence \ $\beta _{\mathbf{n}}$ \
of product type for which 
\begin{equation*}
{\lim_{\mathbf{n}}}\ \frac{\beta _{\mathbf{n}}}{b_{\mathbf{n}}}=0\ \text{\
and }{\sum_{\mathbf{n}}}\frac{a_{\mathbf{n}}}{\beta _{\mathbf{n}}^{r}}%
<+\infty .
\end{equation*}
\end{lemma}

\noindent Here is our general strong law of large numbers.

\begin{theorem}
\label{theo4} Let $a_{\mathbf{n}}$ be a non-negative 
d-sequence and let $r>0$. Suppose that $b_{\mathbf{n}}$ is a positive,
non-decreasing, unbounded d-sequence of product type. If
\begin{equation*}
\sum_{\mathbf{n}}\ \frac{a_{\mathbf{n}}}{b_{\mathbf{n}}^{r}}<+\infty
\end{equation*}%
and there exists $C>0$ such that for any $\mathbf{n}\in \mathbf{N}^{d}$ and
any $\varepsilon >0$ 
\begin{equation*}
\mathbb{P}\left( \max_{\mathbf{m}\leq \mathbf{n}}|S_{\mathbf{m}}|\geq
\varepsilon \right) \leq C\ \varepsilon ^{-r}\sum_{\mathbf{m}\leq \mathbf{n}%
}a_{\mathbf{m}}
\end{equation*}%
then 
\begin{equation*}
\lim_{\mathbf{n}\rightarrow +\infty }\frac{S_{\mathbf{n}}}{b_{\mathbf{n}}}%
=0.\ \ a.s
\end{equation*}%

\end{theorem}

\section{Proofs of the main results}

\label{sec3}

We will need Lemma \ref{lemma1} and these two following lemmas

\begin{lemma}
\label{lemma2} \ Let \ $\{Y_{\mathbf{k}}\ ,\ \ \mathbf{k}\in \mathbb{\mathbf{%
N}}^{d}\}$ \ be a field of random variables defined on a fixed probability
space \ $(\Omega ,\mathcal{F},\mathbb{P}).$ \ Then for all \ $x\in \mathbb{%
\mathbf{R}}$, 
\begin{equation*}
\mathbb{P}\left( {\sup_{\mathbf{k}}}Y_{\mathbf{k}}>x\right) ={\lim_{\mathbf{n%
}\longrightarrow \infty }}\ \mathbb{P}\left( {\max_{\mathbf{k}\leq \mathbf{n}%
}}Y_{\mathbf{k}}>x\right)
\end{equation*}
\end{lemma}

\begin{proof}
It is easy to see that, for all \ $x\in \mathbb{\mathbf{R}}.$ 
\begin{equation*}
\left( {\sup_{\mathbf{k}}}Y_{\mathbf{k}}>x\right) ={\bigcup_{\mathbf{n=1}%
}^{\infty }}\ \left( {\max_{\mathbf{k}\leq \mathbf{n}}}Y_{\mathbf{k}%
}>x\right)
\end{equation*}%
Hence, by the monotone convergence Theorem for probabilities, we get the
statement.
\end{proof}

\begin{lemma}
\label{lemma3} Let \ $\{Y_{\mathbf{k}},\ \mathbf{k}\in \mathbb{\mathbf{N}}%
^{d}\}$ \ be a field of random variables defined on a fixed probability
space \ $(\Omega ,\mathcal{F},\mathbb{P}).$ \ and \ $\{\varepsilon _{\mathbf{%
n}},\ \mathbf{n}\in \mathbb{\mathbf{N}}^{d}\}$ \ a nondecreasing field of
real numbers. If 
\begin{equation*}
{\lim_{\mathbf{n}\longrightarrow \infty }}\mathbb{P}\left( {\sup_{\mathbf{k}}%
}Y_{\mathbf{k}}>\varepsilon _{\mathbf{n}}\right) =0,
\end{equation*}%
then \ ${\sup_{\mathbf{k}}}Y_{\mathbf{k}}<\infty $ \ a.s.
\end{lemma}

\begin{proof}
By using the monotone convergence Theorem for probabilities, we have

\begin{equation*}
\mathbb{P}\left( {\bigcap_{\mathbf{n=1}}^{\infty }}\left( {\sup_{\mathbf{k}}}%
Y_{\mathbf{k}}>\varepsilon _{\mathbf{n}}\right) \right) ={\lim_{\mathbf{n}%
\longrightarrow \infty }}\mathbb{P}\left( {\sup_{\mathbf{k}}}Y_{\mathbf{k}%
}>\varepsilon _{\mathbf{n}}\right) =0
\end{equation*}

\noindent which is equivalent to \ $\mathbb{P}\left( \bigcup_{\mathbf{n=1}%
}^{\infty }\left( \sup_{\mathbf{k}}\ Y_{\mathbf{k}} \leq \varepsilon _{\mathbf{n}}\right) \right) =1$\newline
This implies that there exists \ $\textbf{n}_{\omega }\in \textbf{N}^{d}$ for almost every \ $%
\omega \in \Omega $ \ such that $\sup_{\mathbf{k}}\ Y_{\mathbf{k}}(\omega
)\leq \varepsilon _{\mathbf{n}_{\omega }}<\infty .$
\end{proof}

\bigskip

\noindent We need more notation for the proofs. In \ $\textbf{N}^{d},$ \ the
maximum is defined coordinate-wise (actually we shall use it only for
rectangles). If \ $\textbf{n}=(n_{1},...,n_{d})\in \mathbf{N}^{d},$ \ then \ 
$\langle \textbf{n} \rangle=\prod_{i=1}^{d}n_{i}.$ A numerical sequence \ $a_{\textbf{n}},\ \textbf{n}\in
\textbf{N}^{d}$ is called d-sequence. If \ $a_{\mathbf{n}}$ \ is a d-sequence
then its difference sequence, i.e. the d-sequence \ $b_{\mathbf{n}}$ \
for which $\sum_{\mathbf{m}\leq \mathbf{n}}b_{\mathbf{m}}=a_{\mathbf{n}},\
\textbf{n}\in \textbf{N}^{d},$ \ will be denoted by \ $\Delta \ a_{\mathbf{n}}(i.e.\ \Delta \
a_{\mathbf{n}}=b_{\mathbf{n}}$). We shall say that a d-sequence \ $a_{%
\mathbf{n}}$ \ is of product type if \ $a_{\mathbf{n}}=%
\prod_{i=1}^{d}a_{n_{i}}^{(i)},$ \ where \ $a_{n_{i}}^{(i)}(n_{i}=0,1,2,...)$
\ is a (single) sequence for each \ $i=1,...d.$ \ Our consideration will be
confined to normalizing constants of product type : \ $b_{\mathbf{n}}$ \
will always denote \ $b_{\mathbf{n}}=\prod_{i=1}^{d}b_{n_{i}}^{(i)},$ \
where \ $b_{n_{i}}^{(i)}(n_{i}=0,1,2,...)$ \ is a nondecreasing sequence of
positive numbers for each \ $i=1,...,d.$ \ In this case we shall say that \ $%
b_{\mathbf{n}}$ \ is a positive nondecreasing \ d-sequence of product
type. Moreover, if for each \ $i=1,...,d$ \ the sequence \ $b_{n_{i}}^{(i)}$
\ is unbounded, then \ $b_{\mathbf{n}}$ \ is called positive, nondecreasing,
unbounded d-sequence of product type. \ As usual, $\log ^{+}(x):\doteq
\max (1,\log (x))$, $x>0$ \and\ $|\log \textbf{n}|:=\prod_{m=1}^{d%
}\log ^{+}n_{m}\ .$ \newline

\noindent \textbf{Proof of Proposition \ref{prop2}}. It is clear that (ii)
implies (i) by taking $b_{m_{j}}=1$ for all $\mathbf{m}\in \mathbf{N}^{d}$
and $1\leq j\leq d$. Now we turn to $(i)\Longrightarrow (ii).$ 
We can assume without loss of generality that $b_{0,j}=1$ for $1\leq j\leq d.$ \
If not, we would replace $b_{\mathbf{m}}$ by $%
\prod_{j=1}^{d}b_{m,j}/b_{0,j},$ $\mathbf{m}\in \mathbf{N}^{d}$\ and $(ii)$
would remain true with a new constant equal to $Cb_{0}^{-r}=C(%
\prod_{j=1}^{d}b_{\mathbf{0},j})^{-r}.$ Now consider a fixed \ $\mathbf{n}%
\in \mathbf{N}^{d}$ \ and an arbitrary a real number $c>1$. Remark
by the monotonicity of $(b_{\mathbf{m}})$ that $b_{\mathbf{m}_{j}}\geq 1$ for
all $\mathbf{m}\in N^{d}$ and that the sequence $(c^{p})_{p\geq 0}$ forms a
partition of $[1,+\infty \lbrack .$ This implies that for any $\mathbf{m}\in 
\mathbf{N}^{d}$, for any $1\leq j\leq d,$ there exists a nonnegative integer 
$i_{j}$ such that $c^{i_{j}}\leq b_{m_{j}}<c^{i_{j}+1}.$ Thus for $\mathbf{i}%
=(i_{1},...,i_{d}),$ we have that $\mathbf{m}\in \mathcal{A}_{\mathbf{i}}=\{%
\mathbf{s}\in \mathbb{\mathbf{N}}^{d}\ \ \ and\ \ \ c^{i_{j}}\leq b_{\mathbf{%
s}_{j}}<c^{i_{j}+1},j=1,...,d\}.$ Since this holds for all $\mathbf{m}\in 
\mathbf{N}^{d},$ we get 
\begin{equation*}
\mathbf{N}^{d}=\bigcup_{i\in \mathbf{N}^{d}}\mathcal{A}_{\mathbf{i}}.
\end{equation*}
Let us restrict ourselves to $\mathbf{m}\leq \mathbf{n},$ and let us define 
\begin{equation*}
\mathcal{A}_{\mathbf{i,n}}=\{\mathbf{s}\in \mathbb{\mathbf{N}}^{d},\mathbf{s}%
\leq \mathbf{n}\ \ \ and\ \ \ c^{i_{j}}\leq b_{\mathbf{s}%
_{j}}<c^{i_{j}+1},j=1,...,d\}.
\end{equation*}
Since $c^{p}\rightarrow \infty $ as $p\rightarrow \infty $ and for $\mathbf{m%
}\in $ $\mathcal{A}_{\mathbf{i,n}},$ for $1\leq j\leq d$, $b_{\mathbf{s}%
_{j}}\leq b_{\mathbf{n}_{j}}\leq \max \{b_{n_{k}},1\leq k\leq d\}<\infty ,$
the sets $\mathcal{A}_{\mathbf{i,n}}$ are empty for large values of $i.$
Then put $\mathbf{k}_{\mathbf{n}}=\max \{\mathbf{i}:\mathcal{A}_{\mathbf{i,n}%
}\neq \emptyset \}<+\infty $ and we have 
\begin{equation*}
\lbrack 0,n]=\bigcup_{\mathbf{i}\leq \mathbf{k}_{\mathbf{n}}}\mathcal{A}_{%
\mathbf{i,n}}.
\end{equation*}
It is also noticeable that if $\mathbf{m}\leq \mathbf{s}\in $\ $\mathcal{A}_{%
\mathbf{i,n}}$, then necessarily $\mathbf{m}$ is in some $\mathcal{A}_{%
\mathbf{i}^{\prime }\mathbf{,n}}$ with $\mathbf{i}^{\prime }\leq \mathbf{i}.$
As well let $\mathbf{m}_{\mathbf{i,n}}=\max \mathcal{A}_{\mathbf{i,n}}\leq 
\mathbf{n}$ and define $D_{\mathbf{i,n}}={\sum_{\mathbf{m}\in \mathcal{A}_{%
\mathbf{i,n}}}}a_{\mathbf{m}}$\ where, by convention, $D_{\mathbf{i,n}}=0$
and $\mathbf{m_{i,n}}=(0,...,0)$ when $\mathcal{A}_{\mathbf{i,n}}=\emptyset $%
. From all that, we have

\begin{equation*}
\mathbb{P}\left( {\max_{\mathbf{m}\leq \mathbf{n}}}|S_{\mathbf{m}}|b_{%
\mathbf{m}}^{-1}\geq \varepsilon \right) \leq {\sum_{\mathbf{i}\leq \mathbf{k%
}_{\mathbf{n}}}}\mathbb{P}\left( {\max_{\mathbf{m}\in \mathcal{A}_{\mathbf{%
i,n}}}}|S_{\mathbf{m}}|b_{\mathbf{m}}^{-1}\geq \varepsilon \right) .
\end{equation*}
Since for $\mathbf{m}\in \mathcal{A}_{\mathbf{i,n}},$ $b_{\mathbf{m}%
}=\prod_{j=1}^{d}b_{m_{j}}\geq \prod_{j=1}^{d}c^{i_{j}}$ and $\mathcal{A}_{%
\mathbf{i,n}}\subset \lbrack 0,m_{\mathbf{i,n}}],$ we get

\begin{equation*}
\mathbb{P}\left( {\max_{\mathbf{m}\leq \mathbf{n}}}|S_{\mathbf{m}}|b_{%
\mathbf{m}}^{-1}\geq \varepsilon \right) \leq {\sum_{\mathbf{i}\leq \mathbf{k%
}_{\mathbf{n}}}}\mathbb{P}\left( {\max_{\mathbf{m}\in \mathcal{A}_{\mathbf{%
i,n}}}}|S_{\mathbf{m}}|b_{\mathbf{m}}^{-1}\geq \varepsilon \right) .\leq {%
\sum_{\mathbf{i}\leq \mathbf{k}_{\mathbf{n}}}}\ \mathbb{P}\left( {\max_{%
\mathbf{m}\in \mathcal{A}_{\mathbf{i,n}}}}|S_{m}|\geq \varepsilon {%
\prod_{j=1}^{d}}c^{i_{j}}\right) .
\end{equation*}
Now by applying (i) one arrives at

\begin{equation*}
\mathbb{P}\left( {\max_{\mathbf{m}\leq \mathbf{n}}}|S_{\mathbf{m}}|b_{%
\mathbf{m}}^{-1}\geq \varepsilon \right) \leq C\varepsilon ^{-r}{\sum_{%
\mathbf{i}\leq \mathbf{k}_{\mathbf{n}}}\prod_{j=1}^{d}c^{-ri_{j}}\sum_{%
\mathbf{m\leq m}_{i,n}}}a_{m}\leq C\varepsilon ^{-r}{\sum_{\mathbf{i}\leq 
\mathbf{k}_{\mathbf{n}}}\prod_{j=1}^{d}}c^{-ri_{j}}\sum_{\mathbf{m}\leq 
\mathbf{i}}D_{\mathbf{m,n}}.
\end{equation*}
By the remark made above, $\mathbf{m\leq m}_{i,n}\in $ $\mathcal{A}_{\mathbf{%
i,n}}$ implies that $\mathbf{m}$ is in some $\mathcal{A}_{\mathbf{s,n}}$
where $\mathbf{s\leq i}$ and then by the definition of the $D_{\mathbf{i,n}}$
on has ${\sum_{\mathbf{m\leq m}_{i,n}}}a_{m}\leq \sum_{\mathbf{m}\leq 
\mathbf{i}}D_{\mathbf{m,n}}$ and next\bigskip 
\begin{equation*}
\mathbb{P}\left( {\max_{\mathbf{m}\leq \mathbf{n}}}|S_{\mathbf{m}}|b_{%
\mathbf{m}}^{-1}\geq \varepsilon \right) \leq C\varepsilon ^{-r}{\sum_{%
\mathbf{i}\leq \mathbf{k}_{\mathbf{n}}}\prod_{j=1}^{d}}c^{-ri_{j}}\sum_{%
\mathbf{m}\leq \mathbf{i}}D_{\mathbf{m}},
\end{equation*}
which becomes by a straightforward manipulations on the ranges of the sums, and
where $k_{\mathbf{n}}(j)$ stands for the $j$-$th$ coordinate of $k_{\mathbf{n%
}},$

\begin{equation*}
\mathbb{P}\left( {\max_{\mathbf{m}\leq \mathbf{n}}}|S_{\mathbf{m}}|b_{%
\mathbf{m}}^{-1}\geq \varepsilon \right) \leq C\varepsilon ^{-r}{\sum_{m\leq 
\mathbf{k}_{\mathbf{n}}}D_{\mathbf{m,n}}\sum_{m\leq \mathbf{i}\leq \mathbf{k}%
_{\mathbf{n}}}\prod_{j=1}^{d}}c^{-ri_{j}}
\end{equation*}

\begin{equation*}
\leq C\varepsilon ^{-r}{\sum_{m\leq \mathbf{k}_{\mathbf{n}}}}D_{\mathbf{m,n}}%
{\prod_{j=1}^{d}\sum_{m_{j}\leq i_{j}\leq k_{\mathbf{n}}(j)}}c^{-ri_{j}}
\end{equation*}

\begin{equation*}
=C\varepsilon ^{-r}{\sum_{m\leq \mathbf{k}_{\mathbf{n}}}}D_{\mathbf{m,n}}{%
\prod_{j=1}^{d}}\frac{c^{-rm_{j}}-c^{-r(k_{\mathbf{n}}(j)+1)}}{1-c^{-r}}\leq
C\varepsilon ^{-r}{\sum_{m\leq \mathbf{k}_{\mathbf{n}}}D_{\mathbf{m,n}%
}\prod_{j=1}^{d}}\frac{c^{-rm_{j}}}{1-c^{-r}},
\end{equation*}
since $c>1$ and $k_{\mathbf{n}}(j)+1>m_{j}.$ Now, at this last but one step,
we have

\begin{equation*}
\mathbb{P}\left( {\max_{\mathbf{m}\leq \mathbf{n}}}|S_{\mathbf{m}}|b_{%
\mathbf{m}}^{-1}\geq \varepsilon \right) \leq C\varepsilon ^{-r}\left( \frac{%
c^{r}}{1-c^{-r}}\right) ^{d}\text{ }{\sum_{m\leq \mathbf{k}_{\mathbf{n}}}D_{%
\mathbf{m,n}}\prod_{j=1}^{d}}c^{-r(m_{j}+1)}
\end{equation*}

\begin{equation*}
\leq C\varepsilon ^{-r}\left( \frac{c^{r}}{1-c^{-r}}\right) ^{d}{\sum_{m\leq 
\mathbf{k}_{\mathbf{n}}}\sum_{\mathbf{s}\in \mathcal{A}_{\mathbf{m,n}}}}a_{%
\mathbf{s}}{\prod_{j=1}^{d}}c^{-r(m_{j}+1)}.
\end{equation*}

\noindent Finally, taking into account the fact that for $\mathbf{s}\in \mathcal{A}_{\mathbf{m,n%
}},$ $c^{m_{j}+1}\geq b_{s_{j}}$, $1\leq j\leq d,$ that is ${%
\prod_{j=1}^{d}}c^{r(m_{j}+1)}\geq b_{\mathbf{s}}^{r},$ we arrive at 
\begin{equation*}
\mathbb{P}\left( {\max_{\mathbf{m}\leq \mathbf{n}}}|S_{\mathbf{m}}|b_{%
\mathbf{m}}^{-1}\geq \varepsilon \right) \leq C\varepsilon ^{-r}\left( \frac{%
c^{r}}{1-c^{-r}}\right) ^{d}\text{ }{\sum_{m\leq \mathbf{k}_{\mathbf{n}%
}}\sum_{\mathbf{s}\in \mathcal{A}_{\mathbf{m,n}}}}\frac{a_{\mathbf{s}}}{b_{%
\mathbf{s}}^{r}}
\end{equation*}

\begin{equation*}
\leq C\varepsilon ^{-r}\left( \frac{c^{r}}{1-c^{-r}}\right) ^{d}\text{ }{%
\sum_{\mathbf{m}\leq \mathbf{n}}}\frac{a_{\mathbf{m}}}{b_{\mathbf{m}}^{r}}.
\end{equation*}

\noindent Since \ $c$ is arbitrary $c>1$ and $\min_{c>1}\frac{c^{r}}{1-c^{-r}}=4,$ we
achieve the proof by 
\begin{equation*}
\mathbb{P}\left( {\max_{\mathbf{m}\leq \mathbf{n}}}|S_{\mathbf{m}}|b_{%
\mathbf{m}}^{-1}\geq \varepsilon \right) \leq 4^{d} \ C \ \ \varepsilon ^{-r}\text{ }%
{\sum_{\mathbf{m}\leq \mathbf{n}}}\frac{a_{\mathbf{m}}}{b_{\mathbf{m}}^{r}}.
\end{equation*}

\bigskip

\noindent \textbf{Proof of Theorem \ref{theo4}}. Let $\beta _{\mathbf{n}}$
be the d-sequence obtained in the Lemma \ref{lemma1}. According to
Proposition \ref{prop2} 
\begin{equation*}
\forall \mathbf{m}\leq \mathbf{n},\ \ \ \mathbb{P}\left( \max_{\mathbf{\ell}\leq 
\mathbf{m}}|S_{\mathbf{\ell}}|\ \beta _{\mathbf{\ell}}^{-1}\geq \varepsilon_{\mathbf{k}}\right) \leq 4^{d}\ C\ \varepsilon_{\mathbf{k}}^{-r}\ \sum_{\mathbf{\ell}\leq 
\mathbf{m}}a_{\mathbf{\ell}}\ \beta _{\mathbf{\ell}}^{-r}\ \ .
\end{equation*}

\noindent By this fact we get for any fixed $\mathbf{k}\in \textbf{N}^{d}$

\begin{equation*}
\mathbb{P}\left( \sup_{\mathbf{\ell}\leq \mathbf{m}}|S_{\mathbf{\ell}}|\ \beta _{\mathbf{\ell}}^{-1}\geq \varepsilon _{\mathbf{k}}\right)
\end{equation*}

\begin{equation*}
\leq \lim_{\mathbf{m}\rightarrow +\infty }\ \mathbb{P}\left( \max_{\mathbf{\ell}%
\leq \mathbf{m}}|S_{\mathbf{\ell}}|\ \beta _{\mathbf{\ell}}^{-1}\geq \varepsilon _{%
\mathbf{k}}\right)
\end{equation*}

\begin{equation*}
 \leq 4^{d}\ C\ \varepsilon _{\mathbf{k}}^{-r}\ \sum_{\mathbf{n}}\ a_{\mathbf{%
n}}\ \beta _{\mathbf{n}}^{-r}
\end{equation*}

\noindent where $\{\varepsilon _{\mathbf{k}},\ \textbf{k}\in \textbf{N}^{d}\}$ \ a positive,
non decreasing, unbounded field of real numbers. So we have by Lemma \ref%
{lemma1}

\begin{equation*}
\lim_{\mathbf{k}\rightarrow +\infty }\ \mathbb{P}\left( \sup_{\mathbf{\ell}}|S_{%
\mathbf{\ell}}|\ \beta _{\mathbf{\ell}}^{-1}\geq \varepsilon _{\mathbf{k}}\right)
=0.
\end{equation*}%
Using Lemma \ref{lemma2} 
\begin{equation*}
\mathbb{P}\left( \sup_{\mathbf{\ell}}|S_{\mathbf{\ell}}|\ \beta _{\mathbf{\ell}%
}^{-1}\geq \varepsilon _{\mathbf{k}}, \ \ for \ \ all \ \ \textbf{k}\in \textbf{N}^{d}\right) =0.
\end{equation*}%
So we have by Lemma \ref{lemma3}%
\begin{equation*}
\sup_{\mathbf{\ell}}|S_{\mathbf{\ell}}|\ \beta _{\mathbf{\ell}}^{-1}<\infty \ \ a.s\ .
\end{equation*}%
Finally by Lemma \ref{lemma1}%
\begin{equation*}
0\leq \frac{|S_{\mathbf{n}}|}{b_{\mathbf{n}}}=\frac{|S_{\mathbf{n}}|}{\beta
_{\mathbf{n}}}  \\\ \frac{\beta _{\mathbf{n}}}{b_{\mathbf{n}}}
\end{equation*}%
\begin{equation*}
\leq \sup_{\mathbf{\ell}}|S_{\mathbf{\ell}}| \\\ \beta _{\mathbf{\ell}}^{-1} \\\ \frac{\beta _{\mathbf{n}}}{b_{\mathbf{n}}}\rightarrow
0. \ a.s
\end{equation*}%

\section{Conclusion}

\label{sec4}

\subsection{A first Application : Logarithmically weighted sums}

The following result is an extension of Theorem 7 in Nosz\'{a}ly and T\'{o}m\'{a}%
cs \cite{noszaly} and of Theorem 4.2 in Fazekas et al. \cite{fazekas99}. In
this Theorem, we do not need any moment assumption in contrary of these
above cited theorems.

\begin{theorem}
\label{theo5} Let \ $\{X_{\mathbf{n}},\mathbf{n}\in \mathbb{\mathbf{N}}%
^{d}\} $ \ be a field of random variables. Let \ $r>1$ . We assume there
exists \ $C>0$ \ such that for any \ $\mathbf{m} \in \mathbb{\mathbf{N}}%
^{d} $ \ and any \ $\varepsilon >0$ 
\begin{equation*}
\mathbb{P}\left( {\max_{\mathbf{\ell}\leq \mathbf{m}}\sum_{\mathbf{k}\leq \ 
\mathbf{\ell}}}\frac{X_{\mathbf{k}}}{\langle\mathbf{k}\rangle}\geq \varepsilon \right) \leq
C\varepsilon ^{-r}{\sum_{\ \mathbf{\ell}\leq \mathbf{m}}}\frac{1}{\langle\ \mathbf{\ell}\rangle%
}.
\end{equation*}

\noindent Then 
\begin{equation*}
\frac{1}{|\log \mathbf{n}|}{\sum_{\mathbf{k}\leq \mathbf{n}}}\frac{X_{\mathbf{k%
}}}{\langle \mathbf{k} \rangle}\longrightarrow 0\ (\mathbf{n}\longrightarrow +\infty \ )\
a.s
\end{equation*}
\end{theorem}

\begin{proof}
Let us apply Theorem \ref{theo4} with $a_{\mathbf{n}} = \frac{1}{\langle%
\mathbf{n} \rangle}\ and\ b_{\mathbf{n}} = |\log \textbf{n}|$. The proof is achieved by
remarking that for $r>1$ 
\begin{equation*}
\sum_{\textbf{n}} \frac{a_{\textbf{n}}}{b_{\textbf{n}}^{r}}= \sum_{\mathbf{n}}\frac{1}{|\log \mathbf{n}%
|^{r}}\frac{1}{ \langle\mathbf{n} \rangle}<+\infty.
\end{equation*}
\end{proof}

\subsection{A second application.}

By using Markov's Inequality and applying our results (see Theorem \ref%
{theo4}), under the same assumptions in Nosz\'{a}ly and T\'{o}m\'{a}cs \cite%
{noszaly}, we rediscover their results.\\

\noindent \textbf{Acknowledgement}.\\
The paper was finalized while the second author was visiting MAPMO, University of Orl\'eans, France, in 2011. He expresses his warm thanks to responsibles of MAPMO for kind hospitality. The authors also thank the referee for his valuable comments and suggestions.\\

\end{document}